\documentclass[12pt,twoside,reqno]{amsart}
\usepackage[colorlinks=true,citecolor=blue]{hyperref}
\usepackage{mathptmx, amsmath, amssymb, amsfonts, amsthm, mathptmx, enumerate, color}
\usepackage{graphicx}
\usepackage{epstopdf}
\usepackage{comment}
\newtheorem{theorem}{Theorem}[section]
\newtheorem{corollary}{Corollary}[section]
\newtheorem{lemma}{Lemma}[section]
\newtheorem{proposition}{Proposition}[section]
\theoremstyle{definition}
\newtheorem{definition}{Definition}[section]
\newtheorem{example}{Example}[section]
\newtheorem{remark}{Remark}[section]

\newtheorem{problem}{Problem}[section]
\numberwithin{equation}{section}
\begin{document}
\setcounter{page}{1}

\vspace*{1.0cm}
\title[Convex and quasiconvex truncations of nonconvex functions]
{Convex and quasiconvex truncations of nonconvex functions}
\author[C. Pintea]{Cornel Pintea}
\address{
Babe\c s-Bolyai University, Faculty of Mathematics and Computer
Science, Department of Mathematics, Str. Kog\u alniceanu no. 1, 400084 Cluj-Napoca, Romania\\
}
\begin{abstract}
We consider nonconvex real valued functions whose truncations are either quasiconvex or even convex starting with a certain level.
Among them, the $C^2$-smooth functions whose level sets are all completely contained in the positive definite region of their Hessian matrices, 
starting with a certain level, are good examples of such functions. For such a function we show the injectivity of its restricted gradient to a large subset of the 
positive definite region of its Hessian matrices.
\end{abstract}
\begingroup
\renewcommand{\thefootnote}{}
\footnotetext{The author acknowledges support from the project “Singularities and Applications” - CF132/31.07.2023 funded by the European Union - NextGenerationEU - through Romania’s National
Recovery and Resilience Plan, and support by the grant CNRS-INSMI-IEA-329.}
\endgroup

\maketitle
\vspace*{-0.6cm}

\section{Introduction}
Inspired by the example of the product of square distance functions, we investigate here nonconvex functions whose 
sublevel sets start to be convex with a certain level and evaluate their deviation from (quasi)convexity. The region ${\rm Hess}^+(f)$ of a $C^2$-smooth function 
$f\!:\!\mathbb{R}^n\!\!\rightarrow\!\mathbb{R}$, where the associated Hessian matrix is positive definite, plays an important role in our investigations. 
Indeed, the level sets of a function with bounded complement $\mathbb{R}^n\setminus {\rm Hess}^+(f)$ start to be all contained in ${\rm Hess}^+(f)$ at a certain level, 
which make the associated sublevel sets to be all convex, although the function $f$ might not be even quasiconvex. In this respect we define and investigate the relations 
between the {\em the smallest quasiconvexity level} ${\rm sql}(f)$ and {\em the smallest convexity level} ${\rm scl}(f)$ of a, possibly, 
non-(quasi)convex function $f$ whose truncations $T_q(f)=\max\{q,f\}$  are all (quasi)convex starting with a certain value of $q$. These values signify, for good reasons, measures of the of the function's deviation from (quasi)convexity. 
The values
\[h_{\max}(f):=\max \left(f\big|_{\mathbb{R}^n\setminus {\rm Hess}^+(f)}\right)\mbox{ and }\nu_{\max}:=\max\left(f\big|_{C(f)}\right)\]  play some important role in these relations, whenever the critical set $C(f)$ and 
$\mathbb{R}^n\setminus {\rm Hess}^+(f)$ are bounded. 

The paper is organized as follows: In the second section we define the truncations of a function along with the truncated-(quasi)convex functions. 
For the late two types of functions we define the {\em the smallest convexity level} of a truncated convex  and 
{\em the smallest quasiconvexity level} of a truncated quasiconvex function respectively. We also observe here the first relations between these two quantities of a 
truncated convex function. The third section starts with the observation on the inclusion $f^{-1}\left(h_{\max}(f),+\infty\right)\subseteq {\rm Hess}^+(f)$ 
along with the inequality $h_{\max}(f)\geq \nu_{\max}(f)$, for a $C^2$-smooth function $f:\mathbb{R}^n\longrightarrow\mathbb{R}$ such that 
$\mathbb{R}^n\setminus {\rm Hess}^+(f)$ is bounded. If  the critical set of $f$ is additionally bounded, then the truncation-convexity is alo proved here along with the 
inequality ${\rm sql}(f)\leq {\rm scl}(f)\leq \max\{{\rm sql}(f), h_{\max}(f)\}$. We close this section with two corollaries and the example of the product of two squared
functions whose the smallest convexity level is equal with its smallest quasiconvexity level, which are further equal with the associated $h_{\max}$-value. 
In the fourth section we also prove the injectivity of the restricted gradient $\nabla f$ to the set $f^{-1}({\rm scl}(f),+\infty)$, where $f$ is a
$C^2$-smooth truncated convex function such that $f^{-1}({\rm scl}(f),+\infty)\subseteq{\rm Hess}^+(f)$.
The last section is dedicated to some open questions and several observations and comments.
\section{Convex and quasiconvex truncations}
Before we focus to the main results of the paper, we provide several remarks on the values mention in the introduction. We also state and prove preliminary results, some of which will be used in the sequel.
\begin{definition}
A {\em truncation} of a function $f:\mathbb{R}^n\longrightarrow\mathbb{R}$ is a function of type 
$T_q(f):\mathbb{R}^n\longrightarrow\mathbb{R}$, $T_q(f)(x)=\max\{q,f(x)\}$ for some $q\in\mathbb{R}$. 
Then a truncation $T_{q}(f)$ is quasiconvex if and only if 
the sublevel sets $f^{-1}((-\infty,r])$, $r\geq q$ are all convex. If such a level $q\geq \inf f$ exists, $f$ is said to be {\em truncated quasiconvex}. 
The smallest such level $q\geq \inf f$ will be called {\em the smallest quasiconvexity level} of $f$ and denoted by ${\rm sql}(f)$.
If there exists a level $q\geq \inf f$ such that the truncations $T_{r}(f)$ are all convex for $r\geq q$, then $f$ is said to be {\em truncated convex}. 
The smallest such level $q\geq \inf f$ will be called {\em the smallest convexity level} of $f$ and denoted by ${\rm scl}(f)$.
In spite of their labels, the smallest (quasi)convexity levels might only be achieved under some extra assumption. Their precise formal definitions are
\begin{align}
& {\rm sql}(f)=\inf\{ \ q \ | \ q\geq \inf f\mbox{ and }\  T_{r}(f)\mbox{ is quasiconvex for }r\geq q\}\mbox{ if }f\mbox{ is truncated quasiconvex}\nonumber\\
& {\rm scl}(f)=\inf\{ \ q \ | \ q\geq \inf f\mbox{ and }\  T_{r}(f)\mbox{ is convex for }r\geq q\}\mbox{ if }f\mbox{ is truncated convex}\nonumber
\end{align}
Observe that the truncated convexity implies the truncated quasiconvexity.
\end{definition}
\begin{example}
The smallest convexity and (quasi)convexity level of a convex function $f:\mathbb{R}^n\longrightarrow\mathbb{R}$ is ${\rm scl}(f)={\rm sql}(f)=\inf(f)$ which might not be achived. For example the smallest (quasi)convexity level of a nonconstant liniar functional $f:\mathbb{R}^n\longrightarrow\mathbb{R}$ is 
${\rm scl}(f)={\rm sql}(f)=-\infty$, while the smallest (quasi)convexity level of the function  
$E(x)=e^{x_1+\cdots+x_n}$ is ${\rm scl}(E)\!=\!{\rm sql}(E)\!=\!0$, but none of them is achieved by $f$ and $E$ respectively.
\end{example}
\begin{remark}
\label{rem10.10.2025.1}
\begin{enumerate}
\item Every function $f:\mathbb{R}^n\longrightarrow\mathbb{R}$ which is bounded from above is truncated convex and ${\rm scl}(f)\leq \sup(f)$, as the truncation $T_q(f)$ is the constant, obviously convex, function $q$ for every $q\geq \sup(f)$. Therefore, the interesting truncated convex functions $f:\mathbb{R}^n\longrightarrow\mathbb{R}$ are those for which $\sup(f)=+\infty$.
\item Since convexity implies quasiconvexity, it follows that \[{\rm scl}(f)\geq {\rm sql}(f) \geq \inf(f).\]
Alternatively, we observe that the inequality ${\rm scl}(f)<{\rm sql}(f)$ fails as $f^{-1}(-\infty,q]$ is convex for $q>{\rm scl}(f)$ and there is some $q'<{\rm sql}(f)$
such that $f^{-1}(-\infty,q']$ is not convex.\label{rem11.01.2026.1}
\item If $f:\mathbb{R}^n\longrightarrow\mathbb{R}$ is a bounded bellow convex function, then ${\rm scl}(f)={\rm sql}(f)=\inf f$.
\item If $f,g:\mathbb{R}^n\longrightarrow\mathbb{R}$ are truncated quasiconvex functions, then 
$\max\{f,g\}$ is also truncated quasiconvex and  ${\rm sql}(\max\{f,g\})\leq \max\{ {\rm sql}(f),{\rm sql}(g)\}$. 
Indeed, for 
\[
r\geq \max\{{\rm sql}(f),{\rm sql}(g)\}
\]
we have
\begin{align}
\max\{f,g\}^{-1}((-\infty,r]) & =\{x\in\mathbb{R}^n \ | \ \max\{f(x),g(x)\}\leq r\}\nonumber\\
& = \{x\in\mathbb{R}^n \ | \ f(x)\leq r\mbox{ and }g(x)\leq r\}\nonumber\\
& =f^{-1}((-\infty,r])\cap g^{-1}((-\infty,r]),\nonumber
\end{align}
which is convex as the intersection of two convex sets.
 \item The smallest quasiconvexity level of the truncated convex and lower semicontinuous function $f:\mathbb{R}^n\longrightarrow\mathbb{R}$ is achieved by $f$ whenever $f$ is 
 norm coercive, i.e. \[\displaystyle\lim_{\|x\|\rightarrow+\infty}f(x)=+\infty.\]
 If $f$ is additionally convex, then ${\rm scl}(f)=\min f$, while $\min f<{\rm scl}(f)<+\infty$ otherwise.
\item  If $f:\mathbb{R}^n\longrightarrow\mathbb{R}$ is a $C^2$-smooth function such that $\mathbb{R}^n\setminus {\rm Hess}^+(f)$  is bounded, then $\sup(f)=+\infty$ is $f$ is unbounded from above, and
$f^{-1}\left(h_{\max}(f),+\infty\right)\subseteq {\rm Hess}^+(f)$, where 
\[{\rm Hess}^+(f):=\{x\in\mathbb{R}^n \ | \ H_f(x)\mbox{ is positive definite }\}\] and 
\[
h_{\max}(f):=\max\left(f\big|_{\mathbb{R}^n\setminus {\rm Hess}^+(f)}\right).
\]
Indeed, otherwise $f^{-1}(q)\cap\left(\mathbb{R}^n\setminus {\rm Hess}^+(f)\right)\neq\emptyset$
for some $q>h_{\max}(f)$. If $x\in f^{-1}(q)\cap\left(\mathbb{R}^n\setminus {\rm Hess}^+(f)\right)$, i.e. $f(x)=q$ and $x\in \mathbb{R}^n\setminus {\rm Hess}^+(f)$, then 
\[q=f(x)\leq h_{\max}(f)=\max\left(f\big|_{\mathbb{R}^n\setminus {\rm Hess}^+(f)}\right)<q,\] which is absurd.
 \item \label{rem30.01.2026.1}
If $f:\mathbb{R}^n\longrightarrow\mathbb{R}$ is a $C^2$-smooth truncated convex function, then 
\begin{equation}\label{eq27.03.2021.1}
f^{-1}({\rm scl}(f),+\infty)\subseteq{\rm Hess}^+_0(f),\end{equation} 
where ${\rm Hess}^+_0(f)$ stands for the set 
\[\{x\in\mathbb{R}^n \ | \ H_f(x)\mbox{ is positive semi-definite }\}\]
and $H_f(x)$ is the Hessian matrix of $f$ at $x\in\mathbb{R}^n$.
Indeed, for every $x\in f^{-1}({\rm scl}(f),+\infty)$ there exists an open ball $B_x\subseteq f^{-1}({\rm scl}(f),+\infty)$ such that the restriction 
$f\big|_{B_x}$ is convex as $f\big|_{B_x}=T_{{\rm scl}(f)}(f)\big|_{B_x}$ and $T_{{\rm scl}(f)}(f)$ is convex. Thus, the 
Hessian matrix $H_f(y)$ is positive semi-definite at every point $y\in B_x$. Therefore $B_x\subseteq {\rm Hess}^+_0(f)$ for every $x\in f^{-1}({\rm scl}(f),+\infty)$
and the inclusion \eqref{eq27.03.2021.1} follows easily.
 \item If $f:\mathbb{R}^n\longrightarrow\mathbb{R}$ is not a quasiconvex function, but truncated quasiconvex, then ${\rm sql}(f)$ measures somehow the function's deviation from quasiconvexity. Indeed, for every $\varepsilon>0$ there exists 
$y_{\varepsilon}\!\in\!({\rm sql}(f)\!-\!\varepsilon,{\rm sql}(f))$ such that the sublevel set $f^{-1}(-\infty,y_\varepsilon]$ is not convex and ${\rm sql}(f)$ is the smallest value of $f$ with this property. 
\item If $f:\mathbb{R}^n\longrightarrow\mathbb{R}$ is not a convex but a truncated convex $C^2$-smooth function, then ${\rm scl}(f)$ measures somehow the function's deviation from convexity. Indeed, 
for every $\varepsilon>0$ there exists some $z_{\varepsilon}\!\in\!({\rm scl}(f),{\rm scl}(f)\!+\!\varepsilon)$ such that the restriction of  $f$ to the (convex) sublevel set $f^{-1}(-\infty,z_\varepsilon]$ is not convex and ${\rm scl}(f)$ is the smallest value 
of $f$ with this property.  Certainly, the convexity of the restriction of $f$ to some sublevel set $f^{-1}(-\infty,{\rm scl}(f)\!+\!\varepsilon)$ implies, via \cite[Theorem 4.3.1(i)]{Uruty-Lemarechal}, the positive semidefinitness of the hessian matrix $H_f$ of $f$ all over this open sublevel set. The hessian matrix of $f$ is positive definite over the set $f^{-1}({\rm scl}(f),+\infty)$, via the item \eqref{rem30.01.2026.1}. Therefore $H_f$ is positive semindefinite everywhere, which shows, via \cite[Theorem 4.3.1(i)]{Uruty-Lemarechal} once again, the global convexity of $f$, a contradiction with the initial hypothesis.\label{rem30.01.2026.3}
\end{enumerate}\label{rem30.01.2026.2}
\end{remark}
\begin{proposition}\label{theSCLachived}
If $f:\mathbb{R}^n\longrightarrow\mathbb{R}$ is truncated convex and lower semicontinuous, then $T_{\rm scl}(f)$ is convex.
\end{proposition}
\begin{lemma}
If $(r_n)$ is a sequence such that $r_n\searrow {\rm scl}(f)$ as $n\longrightarrow\infty$, then 
\begin{equation}
{\rm cl}\left({\rm epi}(T_{{\rm scl}(f)}(f))\right)={\rm cl}\left(\displaystyle\bigcup_{n\geq 1}{\rm epi}(T_{r_n}(f))\right).
\end{equation}
\end{lemma}
\begin{proof}
We first observe that 
\[
\displaystyle\bigcup_{n\geq 1}{\rm epi}(T_{r_n}(f))\subseteq {\rm epi}(T_{{\rm scl}(f)}(f)), 
\]
as $T_{r_1}(f)\geq T_{r_2}(f)\geq \cdots\geq T_{r_n}(f)\geq \cdots\geq T_{{\rm scl}(f)}(f)$ and
therefore
\begin{equation}\label{asofepi}
{\rm epi}(T_{r_1}(f))\subseteq  {\rm epi}(T_{r_2}(f))\subseteq\cdots\subseteq  {\rm epi}(T_{r_n}(f))\subseteq \cdots\subseteq {\rm epi}(T_{{\rm scl}(f)}(f)).
\end{equation}
Thus
\begin{equation}\label{inclofclepi}
{\rm cl}\left(\displaystyle\bigcup_{n\geq 1}{\rm epi}(T_{r_n}(f))\right)\subseteq {\rm cl}\left({\rm epi}(T_{{\rm scl}(f)}(f))\right).
\end{equation}
Since  ${\rm epi}(T_{r_1}(f)),  {\rm epi}(T_{r_2}(f)),\ldots,  {\rm epi}(T_{r_n}(f)), \ldots$ are all convex and 
\[
{\rm epi}(T_{r_1}(f))\subseteq  {\rm epi}(T_{r_2}(f))\subseteq\cdots\subseteq  {\rm epi}(T_{r_n}(f))\subseteq \cdots\subseteq,   
\]
it follows that 
\[{\rm cl}\left(\displaystyle\bigcup_{n\geq 1}{\rm epi}(T_{r_n}(f))\right)\] is convex too. For the opposite inclusion we first show that
\begin{equation}\label{inclofclepi-1}
{\rm epi}(T_{{\rm scl}(f)}(f))\subseteq{\rm cl}\left(\displaystyle\bigcup_{n\geq 1}{\rm epi}(T_{r_n}(f))\right),
\end{equation}
which proves the opposite inclusion. In this respect we consider $(x,y)\in {\rm epi}(T_{{\rm scl}(f)}(f))$, i.e. 
$y\geq T_{{\rm scl}(f)}(f)(x)=\max\{f(x),{\rm scl}(f)\}$. If $T_{{\rm scl}(f)}(f)(x)=f(x)>{\rm scl}(f)$ or $y>T_{{\rm scl}(f)}(f)(x)={\rm scl}(f)\geq f(x)$,
then there exists $n_0\geq 1$ such that $y\geq r_n$ for every $n\geq n_0$, as $r_n\searrow {\rm scl}(f)\leq T_{{\rm scl}(f)}(f))(x)$.
Therefore $y\geq\max\{f(x),r_n\}=T_{r_n}(f)(x)$, i.e. $(x,y)\in {\rm epi}(T_{r_n}(f))$ for every $n\geq n_0$, which shows that 
\[(x,y)\in\displaystyle\bigcup_{n\geq 1}{\rm epi}(T_{r_n}(f))\subseteq {\rm cl}\left(\displaystyle\bigcup_{n\geq 1}{\rm epi}(T_{r_n}(f))\right).\]
If $y=T_{{\rm scl}(f)}(f)(x)={\rm scl}(f)\geq f(x)$, then 
\[
(x,r_n)\in {\rm epi}(T_{r_n}(f))\subseteq  \displaystyle\bigcup_{n\geq 1}{\rm epi}(T_{r_n}(f)), \ \forall n\geq 1,
\]
which shows that $(x,y)=\displaystyle\lim_{n\rightarrow\infty}(x,r_n)\in {\rm cl}\left(\displaystyle\bigcup_{n\geq 1}{\rm epi}(T_{r_n}(f))\right)$.
\end{proof}
\begin{proof}[Proof of Proposition \ref{theSCLachived}]
We only need to use the lower semicontinuity of \[T_{{\rm scl}(f)}(f)=\max\{f,{\rm scl}(f)\}\] which is characterized by 
the closedness of ${\rm epi}(T_{{\rm scl}(f)}(f))$. Therefore $T_{{\rm scl}(f)}(f)$ is convex as
\[
{\rm epi}(T_{{\rm scl}(f)}(f))={\rm cl}\left({\rm epi}(T_{{\rm scl}(f)}(f))\right)={\rm cl}\left(\displaystyle\bigcup_{n\geq 1}{\rm epi}(T_{r_n}(f))\right). 
\]
is convex.
\end{proof}
\begin{proposition}\label{prop27.03.2021.1}
If $f:\mathbb{R}^n\longrightarrow\mathbb{R}$ is a $C^1$-smooth truncated convex function such that ${\rm scl}(f)\in {\rm Im}(f)$ is
a regular value of $f$, then ${\rm scl}(f)$ is achieved and 
\begin{align}
& \partial T_{{\rm scl}(f)}(f)(x)=\{(\nabla f)_x\}, \ \forall x\in f^{-1}({\rm scl(f)},+\infty)\label{eq31.10.2025.1}\\
& \partial T_{\rm scl(f)}(x)=\{t(\nabla f)_x \ | \ t\in[0,1]\}, \ \forall x\in f^{-1}({\rm scl}(f)),\label{eq31.10.2025.10}
\end{align}
where 
$\partial T_{{\rm scl}(f)}(f)(x):=\{x^*\in\mathbb{R}^n \ | \ T_{{\rm scl}(f)}(f)(y)\geq T_{{\rm scl}(f)}(f)(x)+\langle y-x,x^*\rangle, \ \forall y\in\mathbb{R}^n\}$
is the {\em subdifferential} of $f$ at $x$.
\end{proposition}
\begin{proof} 
The smallest convexity level ${\rm scl}(f)$ is, indeed, achived via Proposition \ref{theSCLachived}. The equalities \eqref{eq31.10.2025.1}, 
\eqref{eq31.10.2025.10} follow by using the well known 
result on the subdifferential of the maximum of a finite family of functions (see e.g. \cite[Corollary 4.3.2, p. 266]{Uruty-Lemarechal-1})
\[\partial T_{\rm scl(f)}(x)=\partial \max\{f,{\rm scl}(f)\}(x)={\rm conv}\{(\partial f)(x),(\partial T_{\rm scl(f))(x)}\}={\rm conv}\{(\nabla f)(x),0\},\]
as $T_{{\rm scl}(f)}(f)(x)=f(x)$,  for all $x\in f^{-1}({\rm scl(f)},+\infty)$.
\end{proof}
\section{The ${\rm Hess}^+$ region towards truncation convexity}
Following up the discussion in the introduction, we are now going to provide several relations involving the values ${\rm sql}$, ${\rm scl}$ along with $h_{\max}$ and $\nu_{\max}$. The additional assumptions on a quasiconvex function to have both the critical set and the complement of its ${\rm Hess}^+$ region bounded ensure the truncated convexity of the function.
\begin{proposition}\label{prop19.10.2025.1}
If $f:\mathbb{R}^n\longrightarrow\mathbb{R}$ is a norm-coercive $C^2$-smooth function such that $\mathbb{R}^n\setminus {\rm Hess}^+(f)$ and $C(f)$ are bounded, 
then $h_{\max}(f)\geq \nu_{\max}(f)$, where 
\[
\nu_{\max}(f)=\max \left(f\big|_{C(f)}\right).
\]
\end{proposition}
\begin{proof}
Since the critical set $C(f)$ of $f$ is also closed, it follows, combined with its boundedness, that  $C(f)$  is 
compact. Also ${\rm Hess}^+(f)$ is an open set, as follows via the Sylvester criterion.
Since \[C(f)=[{\rm Hess}^+(f)\cap C(f)]\cup[(\mathbb{R}^n\setminus {\rm Hess}^+(f))\cap C(f)],\] it is enough to show that 
that $h_{\max}(f)\geq\sup f\big|_{{\rm Hess}^+(f)\cap C(f)}$ as the inequality \[h_{\max}(f)\geq \max~f\big|_{[(\mathbb{R}^n\setminus {\rm Hess}^+(f))\cap C(f)]}\] is obvious.
In this respect we first observe that ${\rm Hess}^+(f)\cap C(f)$ consists in Morse critical points of index zero, i.e. all these critical points are 
local minima, each of which is an isolated local minima critical point. By using the Morse Lemma, the local behaviour of $f$
close to each of these Morse critical points is given by the square norm of 
the ambient space. The components of the appropriate sublevel sets containing these points are diffeomorphic balls of the ambeint space and the components of the corresponding level sets are diffeomorphic spheres. We now assume that  
\begin{equation}\sup~f\big|_{{\rm Hess}^+(f)\cap C(f)}>h_{\max}(f)=\max\left(f\big|_{\mathbb{R}^n\setminus {\rm Hess}^+(f)}\right)\label{eq09.10.2025.1}\end{equation} and consider 
$p\in C(f)$ such that $f(p)=\max~f\big|_{C(f)}$. Observe that $p\in {\rm Hess}^+(f)\cap C(f)$, 
as $\max~f\big|_{C(f)}$ cannot be achived on $[\mathbb{R}^n\setminus {\rm Hess}^+(f)]\cap C(f)$ under the assumption \eqref{eq09.10.2025.1} and therefore
\[
f(p)=\sup~f\big|_{{\rm Hess}^+(f)\cap C(f)}=\max~f\big|_{{\rm Hess}^+(f)\cap C(f)}.
\]
Consider a connected component $B$ of a suitable sublevel set $f^{-1}(-\infty,q]$, containing $p$, 
which is a diffeomorphic ball for $q>f(p)$ sufficiently close to $f(p)$ along with its boundary $S$, 
a component of $f^{-1}(q)$ and a diffeomorphic sphere. The connectedness of the complement of ${\rm cl}~B$ along with the norm coercivity of $f$ 
show the nonconnectedness of the level set $f^{-1}(q)$ simply by considering a path $\gamma:[0,1]\longrightarrow \mathbb{R}^n\setminus {\rm cl}~B$
such that $f(\gamma(0))<f(p)$ and $f(\gamma(1))>q$. Since $(f(p),+\infty)$ is an interval of regular values, it follows, via the Non-Critical Neck Principle
\cite[p.194]{Palais-Terng}, that all levels of $f$ greater than $f(q)$ are nonconnected, a contradiction with \cite[Theorem 3.7]{Bro-Pi}.
\end{proof}
\begin{theorem}\label{th17.01.2026.1}
Let $f:\mathbb{R}^n\longrightarrow\mathbb{R}$ be a $C^2$-smooth truncated quasiconvex function.
If $C(f)$ and $\mathbb{R}^n\setminus {\rm Hess}^+(f)$ are bounded, then $f$ is truncated convex and the following inequalities hold ${\rm sql}(f)\leq {\rm scl}(f)\leq \max\{{\rm sql}(f), h_{\max}(f)\}$.
\end{theorem}
\begin{proof}
The left hand side inequality is justified by Remark \ref{rem10.10.2025.1}\eqref{rem11.01.2026.1}.
We will show that $T_q(f)$ is convex whenever $q>\max\{{\rm sql}(f), h_{\max}(f)\}$, which ensures the inequality \[{\rm scl}(f)\leq \max\{{\rm sql}(f), h_{\max}(f)\}.\]
In this respect we first recall that $f^{-1}((-\infty,r])$ is convex for every $r\geq q$, as $f$ is truncated quasiconvex and $q>{\rm sql}(f)$, and prove that the restriction 
\[
T_q^{x,y}(f):[0,1]\longrightarrow\mathbb{R}, \ T_q^{x,y}(f)(t)=T_q(f)((1-t)x+ty)=\max\{q,f((1-t)x+ty)\}
\]
is a convex function for $q>\max\{{\rm sql}(f), h_{\max}(f)\}$ and every $x,y\in\mathbb{R}^n$. Indeed, if the entire segment $[xy]=\{(1-t)x+ty\}$ is contained in $f^{-1}(q,+\infty)$,
then \[T_q(f)((1-t)x+ty)=f((1-t)x+ty)\] for every $t\in[0,1]$. Therefore $T_q^{x,y}(f)$ is twice differentiable and
\[
 \cfrac{d^2}{dt^2}T_q(f)((1-t)x+ty)=(y-x)H_f((1-t)x+ty)(y-x)^T>0, \ \forall t\in[0,1],
\]
as $[xy]\subset {\rm Hess}^+(f)$ and $H_f((1-t)x+ty)$ is positive definite for every $t\in[0,1]$. Therefore $T_q(f)((1-t)x+ty)\leq (1-t)T_q(f)(x)+tT_q(f)(y)$ in this case.
We now consider
the case when $[xy]\cap f^{-1}(-\infty,q]\neq\emptyset$ and observe that this intersection of convex sets is a segment, 
say $[xy]\cap f^{-1}(-\infty,q]\neq\emptyset=[ab]$. If $a=x$ and $b=y$, then $x,y\in f^{-1}(-\infty,q]$ and 
$q=T_q(f)(a)=T_q(f)(b)=T_q(f)((1-t)a+tb)=(1-t)T_q(f)(a)+tT_q(f)(b)$ for all $q\in [0,1]$. 
If $x=a\in f^{-1}(-\infty,q]$ and $y\in f^{-1}(q,+\infty)$, then 
\[
T_q^{x,y}(f)(t)=\left\{
\begin{array}{cll}
q & \mbox{ if }t\in[0,t_b]\\
f((1-t)x+ty)& \mbox{ if }t\in[t_b,1],
\end{array}
\right.
\]
where $t_b\in(0,1)$ is such that $b=(1-t_b)x+t_by$, then $T_q^{x,y}(f)$ is convex, as $T_q^{x,y}(f)$ is the maximum of the constant (convex) function 
$q$ and the convex function
$C_q^{x,y}(f):[0,1]\longrightarrow\mathbb{R}$ is defined by
\[
C_q^{x,y}(f)(s)=\left\{
\begin{array}{cll}
q+\langle (\nabla f)_b,y-x\rangle(s-t_b) & \mbox{ if }s\in[0,t_b]\\
f((1-s)x+sy) & \mbox{ if }s\in[t_b,1].
\end{array}
\right.
\] 
The representation $T_q^{x,y}(f)=\max\{q,C_q^{x,y}(f)\}$ holds true as $\langle (\nabla f)_b,y-x\rangle\geq 0$. Indeed, the gradient vector 
$(\nabla f)_b$ is a normal vector to the convex set sublevel set 
$Sf(b)=f^{-1}(-\infty,f(b)]$
through $b=(1-t_b)x+t_by$, 
as it actually generates the normal cone of $Sf(b)$ at the point $b=(1-t_b)x+t_by$ (se e.g. \cite[Theorem 1.3.5, p. 245]{Uruty-Lemarechal}). 
In other words, we have
\begin{align}
\langle (\nabla f)_b,x-b\rangle\leq 0 & \Leftrightarrow\langle (\nabla f)_{(1-t_b)x+t_by},x-(1-t_b)x-t_by\rangle\leq 0 \nonumber\\ 
& \Leftrightarrow \langle (\nabla f)_{(1-t_b)x+t_by},t_b(x-y)\rangle\leq 0\nonumber\\
& \Leftrightarrow \langle (\nabla f)_{(1-t_b)x+t_by},y-x\rangle\geq 0. \nonumber
\end{align}
If $x,y\in f^{-1}(q,+\infty)$, but still $[xy]\cap f^{-1}(-\infty,q]\neq\emptyset=[ab]$, then 
\[
T_q^{x,y}(f)(t)=\left\{
\begin{array}{cll}
f((1-t)x+ty)& \mbox{ if }t\in[0,t_a]\\
q & \mbox{ if }t\in[t_a,t_b]\\
f((1-t)x+ty)& \mbox{ if }t\in[t_b,1],
\end{array}
\right.
\]
where $t_a,t_b\in[0,1]$, $t_a\leq t_b$ are such that $a=(1-t_a)x+t_ay$ and $b=(1-t_b)x+t_by$. Note that in the particular case $a=b\Leftrightarrow t_a=t_b$ 
\[
T_q^{x,y}(f)(t)=T_q^{x,y}(f)(t)=\left\{
\begin{array}{cll}
f((1-t)x+ty)& \mbox{ if }t\in[0,t_a]\\
q & \mbox{ if }t=t_a=t_b\\
f((1-t)x+ty)& \mbox{ if }t\in[t_b,1],
\end{array}
\right.=f((1-t)x+ty), \ \forall t\in[0,1].
\] 
Therefore $T_q^{x,y}(f)$ is twice differentiable and convex in this particular case, as $[0,t_a)\cup(t_a,1]\subset {\rm Hess}^+(f)$ and therefore
\[
 \cfrac{d^2}{dt^2}T_q(f)((1-t)x+ty)=(y-x)H_f((1-t)x+ty)(y-x)^T\geq 0, \ \forall t\in[0,1].
\]
If $t_a<t_b$, then $T_q^{x,y}(f)$ is still convex, as $T_q^{x,y}(f)$ is the maximum of the constant (convex) function $q$ and the convex functions
$K_q^{x,y}(f):[0,1]\longrightarrow\mathbb{R}$, $C_q^{x,y}(f):[0,1]\longrightarrow\mathbb{R}$ is defined by
\[
K_q^{x,y}(f)(s)=\left\{
\begin{array}{cll}
f((1-s)x+sy) & \mbox{ if }s\in[0,t_a]\\
q+\langle (\nabla f)_a,y-x\rangle(s-t_a) & \mbox{ if }s\in[t_a,1]
\end{array}
\right.
\] 
\[
C_q^{x,y}(f)(s)=\left\{
\begin{array}{cll}
q+\langle (\nabla f)_b,y-x\rangle(s-t_b) & \mbox{ if }s\in[0,t_b]\\
f((1-s)x+sy) & \mbox{ if }s\in[t_b,1].
\end{array}
\right.
\] 
The representation $T_q^{x,y}(f)=\max\{K_q^{x,y}(f),q,C_q^{x,y}(f)\}$ holds true as $\langle (\nabla f)_a,y-x\rangle\leq 0$ and $\langle (\nabla f)_b,y-x\rangle\geq 0$. 
While the last inequality was alredy proved before, for the first one we notice that the gradient vector 
$(\nabla f)_a$ is a normal vector to the convex set sublevel set 
$Sf(a)=f^{-1}(-\infty,f(a)]=f^{-1}(-\infty,q]$
through $a=(1-t_a)x+t_ay$, 
as it actually generates the normal cone of $Sf(a)$ at the point $a=(1-t_a)x+t_ay$ (se e.g \cite[Theorem 1.3.5, p. 245]{Uruty-Lemarechal}). 
In other words, we have
\begin{align}
\left\langle (\nabla f)_a,\left(1-\cfrac{t_a+t_b}{2}\right)x+\cfrac{t_a+t_b}{2}y-a\right\rangle\leq 0 
& \Leftrightarrow \left\langle (\nabla f)_a,\cfrac{t_a-t_b}{2}x+\cfrac{t_b-t_a}{2}y\right\rangle\leq 0. \nonumber\\
& \Leftrightarrow \cfrac{t_b-t_a}{2}\left\langle (\nabla f)_a,y-x\right\rangle\leq 0 \nonumber\\
& \Leftrightarrow \left\langle (\nabla f)_a,y-x\right\rangle\leq 0. \nonumber
\end{align}
\end{proof}
\begin{corollary}\label{cor31.10.2025.1}
Let $f:\mathbb{R}^n\longrightarrow\mathbb{R}$ be a $C^2$-smooth truncated convex function such that $\sup(f)=+\infty$. If ${\rm sql}(f)\geq h_{\max}(f)$, then ${\rm sql}(f)={\rm scl}(f)=h_{\max}(f)$.
\end{corollary}
\noindent Note however that the opposite inequality ${\rm sql}(f)\leq h_{\max}(f)$ is the one which works when the convexity of the sublevel sets is controlled by the curvature of their corresponding level sets. Indeed, we have the following:
\begin{corollary}\label{cor31.10.2025.2}
Let $f:\mathbb{R}^n\longrightarrow\mathbb{R}$ be a norm-coercive $C^2$-smooth function, where $n\in\{2,3\}$.
If $\mathbb{R}^n\setminus {\rm Hess}^+(f)$ and $C(f)$ are bounded, then $f$ is truncated convex and ${\rm scl}(f)\leq h_{\max}(f)$.
\end{corollary}
\begin{proof}
 The truncated quasiconvexity of $f$ follows from \cite[Theorem 3.7]{Bro-Pi} in the case $n=2$ and the proof of \cite[Theorem 3.4]{Pi-1} in the case $n=3$.
 In fact these proofs show that $f^{-1}(q)$ is a regular convex curve/surface if $q>\max\left(f\big|_{\mathbb{R}^n\setminus {\rm Hess}^+(f)}\right)=h_{\max}(f)$, namely 
 ${\rm sql}(f)\leq h_{\max}(f)$.
 The inequality ${\rm scl}(f)\leq h_{\max}(f)$ follows via Proposition \ref{prop19.10.2025.1}, as ${\rm scl}(f)\leq \max\{h_{\max}(f),{\rm sql}(f)\}=h_{\max}(f)$ follows from. 
\end{proof}
\begin{example}
${\rm sql}(f_a)={\rm scl}(f_a)=h_{\max}(f_a)=3a^4$, where \[f_a:\mathbb{R}^2\longrightarrow\mathbb{R}, \ f_a(x,y)=(x^2+y^2)^2-2a^2(x^2-y^2).\] 
Note that $f_a=d_{(a,0)}^2\cdot d_{(-a,0)}^2-a^4$, the critical set of 
$f_a$ is $C(f_a)=\{(-a,0),(0,0),(a,0)\}$ and the critical zero level set of $f_a$ is the Bernoulli lemniscate,
where $d_{(\alpha,\beta)}:\mathbb{R}^2\longrightarrow\mathbb{R}$ is defined by $d_{(\alpha,\beta)}(x,y)=\sqrt{(x-\alpha)^2+(y-\beta)^2}$. 
The other nonempty level sets are the so called Cassini's ovals.
Also $(-a,0),(a,0)$ are global minima of $f_a$ and $\min f_a=f_a(\pm a,0)=-a^4$. 
In fact $f_a$ is a Morse function and $(-a,0),(a,0)$ are the critical points of index zero and $(0,0)$ is the only critical point of index one. 
Elementary calculations \cite{PiTo} show that 
\[
{\rm Hess}^+(f_a)=\{(x,y)\in\mathbb{R}^2 | 3(x^2+y^2)^2+2a^2(x^2-y^2)>a^4\}=g_{_{\!\frac{a}{\sqrt{3}}}}^{-1}\left(\frac{a^4}{3},+\infty\right),
\]
where 
\[g_b:\mathbb{R}^2\longrightarrow\mathbb{R}, \ f_a(x,y)=(x^2+y^2)^2+2b^2(x^2-y^2).\]
The boundary of ${\rm Hess}^+(f_a)$ is the compact connected smooth regular curve
\[
\{(x,y)\in\mathbb{R}^2 | 3(x^2+y^2)^2+2a^2(x^2-y^2)=a^4\}=g_{_{\!\frac{a}{\sqrt{3}}}}^{-1}\left(\frac{a^4}{3}\right)
\]
and the open set ${\rm Hess}^+(f_a)$ is the unbounded connected component of 
\[\mathbb{R}^2\setminus\partial{\rm Hess}^+(f_a)=\mathbb{R}^2\setminus g_{_{\!\frac{a}{\sqrt{3}}}}^{-1}\left(\frac{a^4}{3}\right),\] which is therefore nonconvex.
By using the Lagrange multipliers technique one can easily see that
\[\max(f\big|_{\partial {\rm Hess}^+(f_a)})=3a^4.\] 
Therefore
$\max(f\big|_{\partial {\rm Hess}^+(f_a)})>0=\nu_{\max}(f_a)=\max(f_a\big|_{C(f_a)})$, which shows that
\[
h_{\max}(f_a)=\max\{f_a(x) \ | \ x\in\mathbb{R}^2\setminus {\rm Hess}^+(f_a)\}=\max\left\{f_a(x) \ | \ x\in g_{_{\!\frac{a}{\sqrt{3}}}}^{-1}\left(-\infty,\frac{a^4}{3}\right]\right\} =3a^4.
\]
Since the convexity of a connected sublevel set with regular boundary (i.e. regular corresponding level set) is controlled by the sign of the curvature 
of the regular corresponding level set, namely the sign of 
\begin{equation}
\left|
\begin{array}{lll}
(f_a)_{xx} & (f_a)_{xy} & (f_a)_x\\
(f_a)_{yx} & (f_a)_{yy} & (f_a)_y\\
(f_a)_x & (f_a)_y & 0
\end{array}
\right|,\label{curvature}
\end{equation}
which is, up to a multiplicative constant, equal to $3(x^2+y^2)^2-c$. In order to select the convex sublevel set one can use the Lagrange multipliers technique 
for the function $(x^2+y^2)^2$ with respect to the constraint $f_a^{-1}(c)$. Elementary calculations show that the  level set $f_a^{-1}(c)$ (Cassini's ovals) 
is regular, nonempty and its curvature does not change the sign if and only if $c\in(-a^4,0)\cup[3a^4,+\infty)$. We need to exclude the sublevel sets $f_a^{-1}(-\infty,c]$
along with their boundaries $f_a^{-1}(c)$, for $c\in(-a^4,0)$, as they are not connected.
On the other hand, the regular value $3a^4$ of $f_a$ is the first one for which the sublevel set $f_a^{-1}(-\infty,3a^4]$ is 
compact, convex and is bounded by a connected Cassini's oval $f_a^{-1}(3a^4)$, which is a connected convex regular curve. Therefore ${\rm sql}(f_a)=3a^4$ and 
${\rm sql}(f_a)={\rm scl}(f_a)=h_{\max}(f_a)=3a^4$, due to Corollary \ref{cor31.10.2025.1}.
\end{example}
\section{The injectivity of the gradient on a large overlevel set}
In this section we rely on the global injectivity criterion by Gale-Nikaid\^ o \cite{GaleNikaido} which  ensures the global injectivity of  $F :  D\longrightarrow \mathbb{R}^n$, where $D\subseteq \mathbb{R}^n$ is a convex open set,
under the positive definiteness of its Fr\^ echet differentials $(dF)_x$, $x \in D$, namely
\begin{equation}
\label{eq31.01.2026.1}
\langle (dF)_x(y),y\rangle > 0, \ \forall x \in D, \ y\in \mathbb{R}^n\setminus\{0\}.\end{equation}
In the case of the gradient $F:=\nabla f$ of a $C^2$-smooth function $f: \mathbb{R}^n \longrightarrow \mathbb{R}^n$, the condition \eqref{eq31.01.2026.1} is the positive definitness of the hessian matrix $H_f$ of $f$.
\begin{remark}
Let $f:\mathbb{R}^n\longrightarrow\mathbb{R}$ be a $C^2$-smooth function such that $\mathbb{R}^n={\rm Hess}^+(f)$, i.e. $H_f(x)$ is positive definite
for all $x\in\mathbb{R}^n$. Then $f$ is strictly convex and its gradient $\nabla f:\mathbb{R}^n\longrightarrow\mathbb{R}^n$ is one-to-one. 
\end{remark}
While the proof of the Gale-Nikaid\^ o criterion uses the convexity of the domain $D$ of $F$, the ${\rm Hess}^+(f)$ region might be non-convex and therefore the restriction of $\nabla f$ to ${\rm Hess}^+(f)$ might not be injective. The restriction of the gradient $\nabla f$ to a slightly smaller set is injective.
\begin{theorem}\label{Th15.01.2026.1}
Let $f:\mathbb{R}^n\longrightarrow\mathbb{R}$ be a $C^2$-smooth truncated convex function such that ${\rm scl}(f)\in {\rm Im}(f)$ is
a regular value of $f$ and $f^{-1}({\rm scl}(f),+\infty)\subseteq{\rm Hess}^+(f)$. 
Then the restriction \[\nabla f\big|_{f^{-1}({\rm scl}(f),+\infty)}:f^{-1}({\rm scl}(f),+\infty)\longrightarrow\mathbb{R}^n\]
is one-to-one.
\end{theorem}
\begin{proof} We first observe that $T_{{\rm scl}(f)}(f):\mathbb{R}^n\longrightarrow\mathbb{R}$ is a continuous convex function and 
\[T_{{\rm scl}(f)}(f)\big|_{f^{-1}[{\rm scl}(f),+\infty)}=f\big|_{f^{-1}[{\rm scl}(f),+\infty)},\]
showing that $\partial T_{\rm scl}(f)(x)=\{(\nabla f)_{x}\}$ for every $x\in f^{-1}({\rm scl}(f),+\infty)$.
According to Rockafellar \cite[Theorem A]{Rockafellar}, the subdifferential operator $\partial T_{\rm scl}(f):\mathbb{R}^n\longrightarrow\mathbb{R}^n$
is maximal monotone. Equivalently, its inverse $\partial T_{\rm scl}(f)^{-1}$ is also maximal monotone and the inverse images $\partial T_{\rm scl}(f)^{-1}(y)$,
$y\in\mathbb{R}^n$ are all convex sets \cite[p. 340]{Rockafellar}, \cite[p. 3, 9, 16]{Phelps}.
Since $H_f(x)$ is positive definite for every $x\in f^{-1}({\rm scl}(f),+\infty)$, it follows that
$\nabla f\big|_{f^{-1}({\rm scl}(f),+\infty)}$ is a CIP function, i.e. $\nabla f\big|_C$ is one-to-one for every convex subset $C$ of 
$f^{-1}({\rm scl}(f),+\infty)$. Indeed the positive definiteness of the Hessian  $H_f(x)$ is equivalent with the Gale-Nikaido \cite{GaleNikaido} condition 
\[
 \langle (d\nabla f)_x(y),y\rangle>0,\mbox{ for all }y\in\mathbb{R}^n\setminus\{0\},
\]
for the gradient $\nabla f$,
which ensures the injectivity of $\nabla f$ on every convex set $C\subseteq f^{-1}({\rm scl}(f),+\infty)$.
Assume that $(\nabla f)_{x_1}=(\nabla f)_{x_2}$ for some $x_1, x_2\in f^{-1}({\rm scl}(f),+\infty)$ and observe that
$x_1,x_2\in (\partial T_{\rm scl}(f))^{-1}(u)$, where $u$ stands for  $(\nabla f)_{x_1}=(\nabla f)_{x_2}$, 
as $\partial T_{\rm scl}(f)(x)=\{(\nabla f)_{x}\}$ for every $x\in f^{-1}({\rm scl}(f),+\infty)$.
Since $\nabla f\big|_{f^{-1}({\rm scl}(f),+\infty)}$ is a CIP function, we deduce, via the convexity of $(\partial T_{\rm scl}(f))^{-1}(u)$ along with
$\partial T_{\rm scl}(f)(x)=\{(\nabla f)_{x}\}$ for every $x\in f^{-1}({\rm scl}(f),+\infty)$,
that 
\begin{equation}
[x_1x_2]\cap f^{-1}(-\infty,{\rm scl}(f)]\neq\emptyset\mbox{ and }[x_1x_2]\subseteq 
(\partial T_{\rm scl}(f))^{-1}(u), \label{eq17.01.2026.1}
\end{equation}
where $[x_1x_2]:=\{u_t:=(1-t)x_1+tx_2 \ | \ t\in[0,1]\}$. We also consider the closed-open segments $[x_1x_2):=\{u_t:=(1-t)x_1+tx_2 \ | \ t\in[0,1)\}$ along 
with $[x_2x_1):=\{v_t:=(1-t)x_2+tx_1 \ | \ t\in[0,1)\}$ and observe that there exists $0<\varepsilon<1$ such that 
\[[x_1u_s), \ [x_2v_s)\subset f^{-1}({\rm scl}(f),+\infty), \ \forall s\in[0,\varepsilon)\] as $x_1,x_2\in f^{-1}({\rm scl}(f),+\infty)$ 
and $f^{-1}({\rm scl}(f),+\infty)$ is an open set. This shows that 
\[
(\partial T_{\rm scl}(f))(u_s)=\{(\nabla f)_{u_s}\}, \ \forall s\in[0,\varepsilon) 
\]
\[
(\partial T_{\rm scl}(f))(v_s)=\{(\nabla f)_{v_s}\}, \ \forall s\in[0,\varepsilon).
\]
By using the second inclusion of \eqref{eq17.01.2026.1} one gets
$[x_1u_s), [x_2v_s) \subseteq [x_1x_2]\subseteq (\partial T_{\rm scl}(f))^{-1}(u)$,
for all $s\in[0,\varepsilon)$. This shows that 
\[
u=(\nabla f)_{x_1}\in (\partial T_{\rm scl}(f))(u_s)=\{(\nabla f)_{u_s}\}, \ \forall s\in[0,\varepsilon) 
\]
\[
u=(\nabla f)_{x_2}\in (\partial T_{\rm scl}(f))(v_s)=\{(\nabla f)_{v_s}\}, \ \forall s\in[0,\varepsilon) 
\]
and contradicts the local injectivity of $\nabla f$ both in $x_1$ and $x_2$.
\end{proof}
\begin{remark}
Let $f:\mathbb{R}^n\longrightarrow\mathbb{R}$ be a $C^2$-smooth truncated convex function. If ${\rm sql}(f)\geq h_{\max}(f)$, 
then ${\rm sql}(f)={\rm scl}(f)=h_{\max}(f)$, due to Corollary \ref{cor31.10.2025.1}, and the following relations hold
\[f^{-1}({\rm scl}(f),+\infty)=f^{-1}\left(h_{\max}(f),+\infty\right)\subseteq {\rm Hess}^+(f)\]
whenever $\mathbb{R}^n\setminus {\rm Hess}^+(f)$ is bounded. In particular $f^{-1}({\rm scl}(f_a),+\infty)\subseteq {\rm Hess}^+(f_a)$.
One can therefore ask whether the injectivity domain of such a function can be extended from $f^{-1}({\rm scl}(f),+\infty)$ to ${\rm Hess}^+(f)$.
This is obviously not the case for such functions which are additionally Morse functions with multiple local minimum points (critical points of index zero) as these points 
are all contained in ${\rm Hess}^+(f)$ and the gradient vanishes at all these local minima points.
\end{remark}
\begin{remark}
Let $f:\mathbb{R}^n\longrightarrow\mathbb{R}$ be a $C^2$-smooth truncated convex function.
Then the restriction $\nabla f\big|_{f^{-1}[{\rm scl}(f),+\infty)}:f^{-1}[{\rm scl}(f),+\infty)\longrightarrow\mathbb{R}^n$
is monotone and $\langle (\nabla f)_x,x-y\rangle\geq 0$ for every $x\in f^{-1}[{\rm scl}(f),+\infty)$ and every $y\in f^{-1}(-\infty,{\rm scl}(f))$. 
Indeed, $(\nabla f)_x\in \partial T_{\rm scl}(f)$ for every $x\in f^{-1}[{\rm scl}(f),+\infty)$ due to Proposition \ref{prop27.03.2021.1} and obviously 
$\partial T_{\rm scl}(f)=\{0\}$ for every $y\in f^{-1}(-\infty,{\rm scl}(f))$. The first statement follows by using the monotonicity of
the subdifferential operator $\partial T_{\rm scl}(f)$ of the convex function $T_{\rm scl}(f)$ on $f^{-1}[{\rm scl}(f),+\infty)$ and the second one by using the same monotonicity 
with $x\in f^{-1}[{\rm scl}(f),+\infty)$ and $y\in f^{-1}(-\infty,{\rm scl}(f))$.
\end{remark}
\begin{example}\label{ex05.02.2026.1}
For the particular function $f_a:\mathbb{R}^2\longrightarrow\mathbb{R}, \ f_a(x,y)=(x^2+y^2)^2-2a^2(x^2-y^2)$, the overlevel set $f_a^{-1}({\rm scl}(f_a),+\infty)=f_a^{-1}(3a^4,+\infty)$ is the largest one with the property that the restriction 
\[\nabla f_a\big|_{f_a^{-1}(3a^4,+\infty)}:f_a^{-1}(3a^4,+\infty)\longrightarrow\mathbb{R}^2\]
is one-to-one. Indeed $f_a$ satisfy the hypothesis of Theorem \ref{Th15.01.2026.1} and ${\rm scl}(f_a)=3a^4$ as well as $\min(f_a)=-a^4$. Moreover, for $k,c\in[-a^4,3a^4)$, $k<c$ one can easily see that 
\[
\nabla f_a\left(\pm\cfrac{\sqrt{3a^4-c}}{2a},\cfrac{\sqrt{c+a^4}}{2a}\right)=4a(0,\sqrt{c+a^4}),
\]
which shows that the restriction 
of the gradient $\nabla f_a$ to the larger overlevel set $f_a^{-1}(k,+\infty)$ is no longer one-to-one for  $k\in[-a^4,3a^4)$, as
\[
 f_a\left(\pm\cfrac{\sqrt{3a^4-c}}{2a},\cfrac{\sqrt{c+a^4}}{2a}\right)=c>k, \mbox{ i.e. } \left(\pm\cfrac{\sqrt{3a^4-c}}{2a},\cfrac{\sqrt{c+a^4}}{2a}\right)\in f_a^{-1}(k,+\infty).
\]
\end{example}
\begin{remark}\label{Rem15.01.2026.1}
The global injectivity of a nonconvex $C^2$-smooth function $f:\mathbb{R}^n\longrightarrow\mathbb{R}$ with multiple critical points satisfying the requirements of 
Theorem \ref{Th15.01.2026.1} fails as the valence of its gradient is at least the cardinality of the critical set $C(f)$ of $f$. Indeed,
the gradient $\nabla f$ vanishes on $C(f)$. Recall that the valence ${\rm Val}(F)$ of  $F : D \longrightarrow \mathbb{R}^n$, as defined in \cite{Neumann}, 
is \[{\rm Val}(F):=\sup\{{\rm card}~F^{-1}(y) : y \in \mathbb{R}^n \}.\]
\end{remark}
\section{Final remarks and open questions}
This section is devoted to open problems related with this work, which are meant to improve some statements or remove some, possibly, redundant hypotheses. 

We only justified Remark \ref{rem30.01.2026.2}\eqref{rem30.01.2026.3} for $C^2$-smooth functions. The question is whether this remark still works for less regular functions such as for $C^1$-smooth functions or even for continuous functions.
\begin{problem}\label{Prob30.01.2026.0}
Does ${\rm sql}(f)$ measures the function's deviation from quasiconvexity, in the sense of Remark \ref{rem30.01.2026.2}\eqref{rem30.01.2026.3}, for less regular functions such as for $C^1$-smooth functions or even for continuous functions?
\end{problem}
Although we justified the inequality ${\rm sql}(f)\leq h_{\max}(f)$ in the proof of Corollary \ref{cor31.10.2025.2} for $C^2$-smooth  functions 
$f:\mathbb{R}^n\longrightarrow\mathbb{R}$, $n\in\{2,3\}$ such that $\mathbb{R}^n\setminus {\rm Hess}^+(f)$ and $C(f)$ are bounded, we are still 
concerned with the opposite inequality, a hypothesis of Corollary \ref{cor31.10.2025.1}, for $C^2$-smooth truncated convex functions.
\begin{problem}\label{Prob15.01.2026.0}
If $f:\mathbb{R}^n\longrightarrow\mathbb{R}$ is a $C^2$-smooth truncated quasiconvex function such that $C(f)$ and $\mathbb{R}^n\setminus {\rm Hess}^+(f)$ are bounded, then it would be interesting to investigate the inequality ${\rm sql}(f)\geq h_{\max}(f)$. This inequality would produce, due to Theorem \ref{th17.01.2026.1}, 
one further equality, namely ${\rm sql}(f)= {\rm scl}(f)=h_{\max}(f)$.
\end{problem}

In several statements we combined the boundedness hypotheses of the critical set $C(f)$ and $\mathbb{R}^n\setminus {\rm Hess}^+(f)$, associated to a norm-coercive $C^2$-smooth function $f:\mathbb{R}^n\longrightarrow\mathbb{R}$. We wonder however whether the boundedness of the critical set and even the norm-coercivity of $f$ are redundant hypotheses in those statements. 
\begin{problem}\label{Prob01.02.2026.1}It would be interesting to investigate whether the boundedness of  the complement $\mathbb{R}^n\setminus {\rm Hess}^+(f)$, associated to a norm coercive $C^2$-smooth function $f:\mathbb{R}^n\longrightarrow\mathbb{R}$, implies the boundedness of the critical set $C(f)$ of $f$. Can the norm coercivity be also removed from the list of hypothesis and still get the boundedness of $C(f)$ just out of the boundedness of the complement  $\mathbb{R}^n\setminus {\rm Hess}^+(f)$?
\end{problem}
On one hand Theorem \ref{Th15.01.2026.1} provides an overlevel set of a function as an injectivity domain for the function's gradient, and, on the other hand, Example \ref{ex05.02.2026.1} shows that the corresponding overlevel set of the particular function $f_a$ is the largest one with the injectivity property for its gradient $\nabla f_a$.
\begin{problem}
It would be interesting to identify the extra properties needed for a function    $f:\mathbb{R}^n\longrightarrow\mathbb{R}$, subject to the hypotheses of Theorem \ref{Th15.01.2026.1}, to conclude that $f^{-1}({\rm scl}(f),+\infty)$ is the largest overlevel set on which its gradient $\nabla f$ is one-to-one.
\end{problem}
If we remove $f^{-1}({\rm scl}(f),+\infty)$ out of ${\rm Hess}^+(f)$, the remaining part 
can still be investigated from the injectivity point of view of $\nabla f$, but we can only recall now that the restriction of $\nabla f$ to this part is injective on every convex subset of ${\rm Hess}^+(f)\setminus f^{-1}({\rm scl}(f),+\infty)$. As we have already seen in Theorem \ref{Th15.01.2026.1}, 
the gradient $\nabla f$ is injective on the quite large subset $f^{-1}({\rm scl}(f),+\infty)$ of 
${\rm Hess}^+(f)$. Now the question is what happens on 
the remaining part of ${\rm Hess}^+(f)$ for a $C^2$-smooth function $f:\mathbb{R}^n\longrightarrow\mathbb{R}$ subject to the requirements of Theorem \ref{Th15.01.2026.1}. 
\begin{problem}\label{Prob15.01.2026.1} Let $f:\mathbb{R}^n\longrightarrow\mathbb{R}$ 
be a nonconvex $C^2$-smooth truncated convex Morse function such that $C(f)$, $\mathbb{R}^m\setminus{\rm Hess}^+(f)$ are bounded and ${\rm scl}(f)\in {\rm Im}(f)$ is
a regular value of $f$ as well as $f^{-1}({\rm scl}(f),+\infty)\subseteq{\rm Hess}^+(f)$.
It would be interesting to investigate whether the restrictions of the gradient $\nabla f$, to the connected components of
${\rm Hess}^+(f)\setminus f^{-1}({\rm scl}(f),+\infty)$ are one-to-one.
\end{problem}
An affirmative answer towards the Problem \ref{Prob15.01.2026.1}, combined with Theorem \ref{Th15.01.2026.1} and Remark \ref{Rem15.01.2026.1}, would 
produce the upperbound \[{\rm card}~\pi_0({\rm Hess}^+(f)\setminus f^{-1}({\rm scl}(f),+\infty))+1\] for the valence of the restricted gradient 
$\nabla f\big|_{{\rm Hess}^+(f)}$, where $\pi_0(X)$ stands 
for the collection of all connected components of the topological space $X$. In other words an affirmative answer towards the Problem \ref{Prob15.01.2026.1}
would produce the upper-bound below for the valence of $\nabla f\big|_{{\rm Hess}^+(f)}$
\begin{equation}\label{eq15.01.2026.2}
{\rm card}~C(f\big|_{{\rm Hess}^+(f)})\leq {\rm Val}(\nabla f\big|_{{\rm Hess}^+(f)})\leq{\rm card}~\pi_0({\rm Hess}^+(f)\setminus f^{-1}({\rm scl}(f),+\infty))+1,
\end{equation}
while the lower-bound is obvious and was mentioned before in Remark \ref{Rem15.01.2026.1}. Note that the above upper-bound for the valence of the gradient restricted 
to ${\rm Hess}^+(f)$ comes with the observation that every connected component of ${\rm Hess}^+(f)\setminus f^{-1}({\rm scl}(f),+\infty)$ contains 
at most one critical point of $f$. This is the case for $f_a$. Note however that an affirmative answer to Problem \ref{Prob15.01.2026.0} leads to the conclusion that 
the connected component $f^{-1}({\rm scl}(f),+\infty)$ of 
${\rm Hess}^+(f)\setminus f^{-1}({\rm scl}(f),+\infty)$ contains no critical points of $f$. Therefore, an affirmative answer towards Problems \ref{Prob15.01.2026.0}
and \ref{Prob15.01.2026.1} imply
\[{\rm card}~C(f\big|_{{\rm Hess}^+(f)})\leq {\rm card}~\pi_0({\rm Hess}^+(f)\setminus f^{-1}({\rm scl}(f),+\infty)).\]
\begin{problem}\label{Prob15.01.2026.2}
 Let $f:\mathbb{R}^n\longrightarrow\mathbb{R}$ 
be a nonconvex $C^2$-smooth truncated convex function such that ${\rm scl}(f)\in {\rm Im}(f)$ is
a regular value of $f$ and $f^{-1}({\rm scl}(f),+\infty)\subseteq{\rm Hess}^+(f)$. Is it true that 
\begin{equation}\label{eq15.01.2026.30}
{\rm Val}(\nabla f\big|_{{\rm Hess}^+(f)})={\rm card}~C(f\big|_{{\rm Hess}^+(f)})={\rm card}~\pi_0({\rm Hess}^+(f)\setminus f^{-1}({\rm scl}(f),+\infty))?
\end{equation}
\end{problem}
\begin{remark}
The restriction $f\big|_{{\rm Hess}^+(f)}$ is a Morse function with only critical points of index zero (local minima). 
Therefore $C(f\big|_{{\rm Hess}^+(f)})=C(f)\cap {\rm Hess}^+(f)$ is a discrete set. A similar question involving ${\rm Val}(\nabla f)$ and ${\rm card}~C(f)$ 
is only interesting when the entire critical set $C(f)$ is discrete (e.g. when $f$ is a Morse function), as otherwise the set 
$C(f)\cap(\mathbb{R}^n\setminus {\rm Hess}^+(f))$ might be infinite 
uncountable and therefore ${\rm Val}(\nabla f)$ is also infinite uncountable in this case.
\end{remark}
\begin{problem}\label{Prob15.01.2026.3}
 Let $f:\mathbb{R}^n\longrightarrow\mathbb{R}$ 
be a nonconvex $C^2$-smooth truncated convex Morse function such that ${\rm scl}(f)\in {\rm Im}(f)$ is
a regular value of $f$ and $f^{-1}({\rm scl}(f),+\infty)\subseteq{\rm Hess}^+(f)$. Is it true that 
\begin{equation}\label{eq15.01.2026.3}
{\rm Val}(\nabla f)={\rm card}~C(f)={\rm card}~\pi_0({\rm Hess}^+(f)\setminus f^{-1}({\rm scl}(f),+\infty))+1?
\end{equation}
\end{problem}
In the particular case of the function $f_a:\mathbb{R}^2\longrightarrow\mathbb{R}$ the sets ${\rm Hess}^+(f_a)$ along with its complement 
$\mathbb{R}^2\setminus {\rm Hess}^+(f_a)$ and $f_a^{-1}({\rm scl}(f_a),+\infty))=f_a^{-1}(3a^4,+\infty))={\rm ext}f_a^{-1}(3a^4)$ along with the components $C_1, C_2$ 
of ${\rm Hess}^+(f_a)\setminus f_a^{-1}({\rm scl}(f_a),+\infty))={\rm Hess}^+(f_a)\setminus f_a^{-1}(3a^4,+\infty))$
are suggested by the figure below. Therefore the inequalities \eqref{eq15.01.2026.2}, in the particular case of $f_a$, become 
$2\leq {\rm Val}(\nabla f\big|_{{\rm Hess}^+(f_a)})\leq 3$, as $C(f_a\big|_{{\rm Hess}^+(f_a)})=\{(-a,0),(a,0)\}$. 
The injectivity of $\nabla f_a\big|_{\mathbb{R}^2\setminus {\rm Hess}^+(f_a)}$ would further produce the following bouunds
$3\leq {\rm Val}(\nabla f_a)\leq 4$. Note that $\nabla f_a\big|_{\mathbb{R}^2\setminus {\rm Hess}^+(f_a)}$ is a local diffeomorphism 
as its Jacobian matrix, i.e. the hessian matrix of $f_a\big|_{\mathbb{R}^2\setminus {\rm Hess}^+(f_a)}$, is everywhere nonsingular. 
\begin{figure}[ht]
\includegraphics[width=14cm]{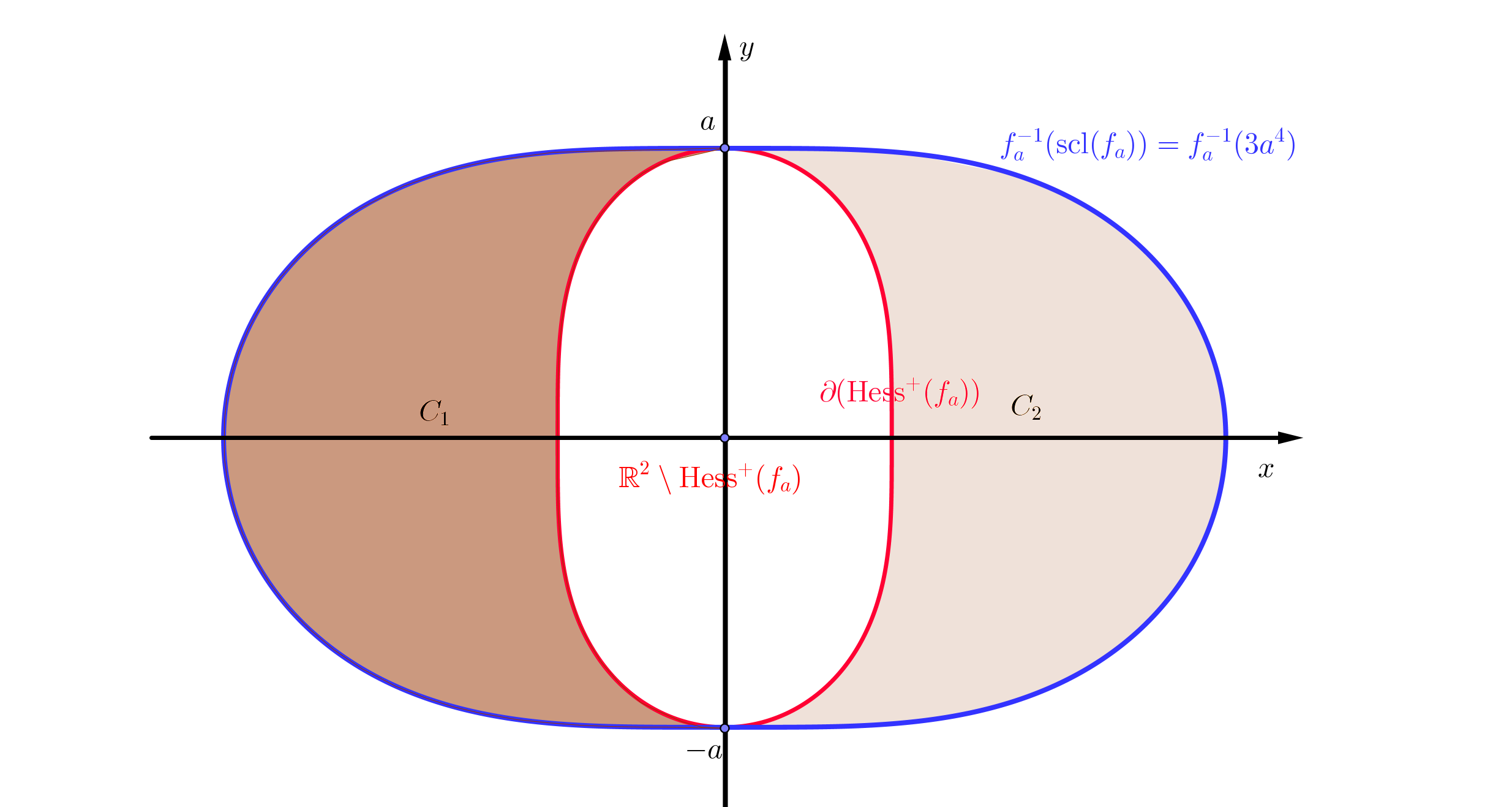}
\caption{}
\end{figure}\\
\noindent On the other hand, once we consider an open subset of $\mathbb{R}^n$,
with nonempty intersection both with ${\rm Hess}^+(f)$ and its complement  $\mathbb{R}^n\setminus{\rm Hess}^+(f)$, the failure of the injectivity of $\nabla f$ is quite high 
as we shall see in the particular case of $f_a$.\\
\noindent We consider the following subsets of $\mathbb{R}^2$ associated to the function $f_a$:
\begin{align}
C^-:=\{(x,y)\in\mathbb{R}^2| \ x<0 \mbox{ and }-a^4\leq f_a(x,y)<0\} \\
C^+:=\{(x,y)\in\mathbb{R}^2| \ x>0 \mbox{ and }-a^4\leq f_a(x,y)<0\} 
\end{align}
whose intersections with ${\rm Hess}^+(f_a)$ and its complement  $\mathbb{R}^n\setminus{\rm Hess}^+(f_a)$ are both nonempty.
\begin{remark}
The restrictions $f_a\big|_{C^{\pm}}$ are not injective as shows the variation of the functions $\|\nabla f_a\big|_{C^{\pm}}\|^2(t,0)=16(t^6-2a^2t^4+a^4 t^2)$
on the intervals $[-a,0)$ and $(0,a]$, via their first order derivative \[\frac{d}{dt}(\|\nabla f_a\big|_{C^{\pm}}\|^2)(t,0))(t,0)=32t(3t^4-4a^2t^2+a^4),\]
as the interval $\{(t,o) \ | \ t\in[-a,0)\}$ is transversal to all level sets of the restriction $f_a\big|_{C^{-}}$ and the interval $\{(t,o) \ | \ t\in(0,a]\}$ 
is transversal to all level sets of $f_a\big|_{C^{+}}$. Since $\nabla f_a$ never vanishes on these level curves, except at the trivial level sets consisting in the 
singleton minimum points $(\pm a,0)$, the normalized gradient vector field \[\cfrac{\nabla f_a}{\|\nabla f_a\|}\]
realizes a diffeomorphism between each nontrivial level curve and the circle $S^1$, while the image of such a level curve through the gradient $\nabla f_a$ 
is an embedded circle surrounding the origin. By fixing such a nontrivial level cuve $\Gamma_k=(f_a\big|_{C^{-}})^{-1}(k)$ for $k\in(-a^4,0)$ of  $f_a\big|_{C^{-}}$ 
we first observe that \[\max\|\nabla f_a\big|_{\Gamma_k}\|^2=16a^4(a^2+\sqrt{a^4+k}).\]
On the other hand $(-\sqrt{a^2+\sqrt{a^4+c}},0),(-\sqrt{a^2-\sqrt{a^4+c}},0)\in\Gamma_c$ and for a suitable choice $k\in(-a^4,0)$ we have 
\begin{align}
& \|\nabla f_a\|^2(-\sqrt{a^2+\sqrt{a^4+c}},0))>\max\|\nabla f_a\big|_{\Gamma_k}\|^2\nonumber\\
& \|\nabla f_a\|^2(-\sqrt{a^2-\sqrt{a^4+c}},0)<\min\|\nabla f_a\big|_{\Gamma_k}\|^2,\nonumber
\end{align}
as
\begin{align}
& \|\nabla f_a\|^2(-\sqrt{a^2+\sqrt{a^4+c}},0))=16a^4(a^2+\sqrt{a^4+c})\stackrel{as \ c\nearrow 0}{\longrightarrow}32a^6\nonumber\\
& \|\nabla f_a\|^2(-\sqrt{a^2-\sqrt{a^4+c}},0))=16a^4(a^2-\sqrt{a^4+c})\stackrel{as \ c\nearrow 0}{\longrightarrow}0.\nonumber
\end{align}
This shows that $\Gamma_c\cap \Gamma_k$ consists in two points at least. One can similarly show that $\Gamma_c'\cap \Gamma_k'$ consists in two points at least, where 
$\Gamma_k'=(f_a\big|_{C^{+}})^{-1}(k)$ for $k\in(-a^4,0)$. These facts show that the valence of $\nabla f_a$ is at least two.
\end{remark}
\begin{proposition}
The sets $C^{\pm}$ are open convex and the restrictions $f_a\big|_{C^{\pm}}$ are quasiconvex.
\end{proposition}
\begin{proof}
The openness of $C_{\pm}$ follow by using their representations as 
\[C_-=p_1^{-1}(-\infty,0)\cap f_a^{-1}(-\infty,0)\mbox{ and }C_+=p_1^{-1}(0,+\infty)\cap f_a^{-1}(-\infty,0),\] 
via the continous functions $p_1$ and $f_a$, where $p_1:\mathbb{R}^2\longrightarrow\mathbb{R}$ is the projection on the first factor.
We recall that the level set $f_a^{-1}(c)$ is, according to \cite{PiTo}, regular, nonempty 
and its curvature does not change the sign if and only if $c\in(-a^4,0)\cup[3a^4,+\infty)$. In fact every sublevel set $f_a^{-1}(-\infty,c]$, 
for $c\in(-a^4,0)$, has two compact connected components $C^{-}_c$, $C^{+}_c$ and is bounded by the regular level $f_a^{-1}(c)$ which is the union of two Cassini's 
ovals whose curvatures do not change their signs and each of these connected components $C^{\pm}_c$ are therefore convex via \cite[Proposition 1, p. 397]{Carmo} combined with 
\cite[Formula 3.7]{Goldman} . These convex components 
can be represented as 
\begin{align}
& C^{-}_c:=\{(x,y)\in\mathbb{R}^2| \ x<0 \mbox{ and }-a^4\leq f_a(x,y)<c\}=(f_a\big|_{C^{-}})^{-1}(-\infty,c)\nonumber \\
& C^{+}_c:=\{(x,y)\in\mathbb{R}^2| \ x>0 \mbox{ and }-a^4\leq f_a(x,y)<c\}=(f_a\big|_{C^{+}})^{-1}(-\infty,c) \nonumber
\end{align}
and obviously 
\begin{equation}
C^-:=\bigcup_{n=1}^\infty C^-_{-\frac{a^4}{n}}\ , C^+:=\bigcup_{n=1}^\infty C^+_{-\frac{a^4}{n}}. \label{eq26.12.2025.1}
\end{equation}
Since 
\[(-a,0)=C^-_{-\frac{a^4}{1}}\subseteq C^-_{-\frac{a^4}{2}}\subseteq \cdots\mbox{ and }(a,0)=C^+_{-\frac{a^4}{1}}\subseteq C^+_{-\frac{a^4}{2}}\subseteq \cdots \]
it follows, via the representations \eqref{eq26.12.2025.1}, that $C_{\pm}$ are convex. The convexity of the sublevel sets $C^{\pm}_c=(f_a\big|_{C^{\pm}})^{-1}(-\infty,c)$ 
show the quasiconvexity of $f_a\big|_{C^{\pm}}$.
\end{proof}
\begin{remark}
The restrictions $f_a\big|_{C^{\pm}}$ are not convex. Indeed by studying the variation of the functions 
$f_a\big|_{C^{-}}(t,0)$ and $f_a\big|_{C^{+}}(t,0)$ on the intervals $(-\sqrt{2}a,0)$ and $(0,\sqrt{2}a)$, where 
\[f_a\big|_{C^{\pm}}(t,0)=t^4-2a^2t^2,\] one can observe, via the second order derivatives $\frac{d^2}{dt^2}(f_a\big|_{C^{\pm}})(t,0)$, that
$f_a\big|_{C^{-}}(t,0)$ and $f_a\big|_{C^{+}}(t,0)$ change the convexity at $-\frac{a}{\sqrt{3}}$ and $\frac{a}{\sqrt{3}}$ respectively.
\end{remark}

\vskip 6mm

\bibliographystyle{amsplain}

\bibliography{references}
\end{document}